\documentclass[12pt]{amsart}

\usepackage{amsmath,amsfonts,amsthm,amssymb}
\usepackage{enumerate}
\usepackage{bm}
\usepackage{upgreek}

\newtheorem{thm}{Theorem}[section]

\newtheorem{lem}[thm]{Lemma}
\newtheorem{pro}[thm]{Proposition}

\theoremstyle{remark}
\newtheorem{rem}[thm]{Remark}

\theoremstyle{definition}

\newcommand{\mC}{\mathbb C}

\newcommand{\mH}{\mathbb H}

\newcommand{\mR}{\mathbb R}

\newcommand{\mZ}{\mathbb Z}

\newcommand{\cP}{\mathcal P}

\newcommand{\cD}{\mathcal D}

\newcommand{\cM}{\mathcal M}

\newcommand{\fg}{\mathfrak g}

\newcommand{\fh}{\mathfrak h}
\newcommand{\fc}{\mathfrak c}
\newcommand{\ft}{\mathfrak t}
\newcommand{\fv}{\mathfrak v}
\newcommand{\fl}{\mathfrak l}

\newcommand{\fC}{\mathfrak C}
\newcommand{\fL}{\mathfrak L}

\newcommand{\rH}{\mathrm H}

\newcommand{\rd}{\mathrm d}

\begin{document}

\title[The contact mappings of a flat $(2,3,5)$-distribution]{The contact mappings of a flat $(2,3,5)$-distribution}
\author{Alex D. Austin}
\date{\today}
\begin{abstract}
Let $\Upomega$ and $\Upomega'$ be open subsets of a flat $(2,3,5)$-distribution. We show that a $C^1$-smooth contact mapping $f : \Upomega \to \Upomega'$ is a $C^\infty$-smooth contact mapping. Ultimately, this is a consequence of the rigidity of the associated stratified Lie group (the Tanaka prolongation of the Lie algebra is of finite-type). The conclusion is reached through a careful study of some differential identities satisfied by components of the Pansu-derivative of a $C^1$-smooth contact mapping. 
\end{abstract}
\maketitle

\section{Introduction}\label{sec:intro}

Let $\cM$ be a connected $C^\infty$-smooth manifold of dimension $n$. A distribution on $\cM$ is a sub-bundle $\mathrm{H}\cM$ of the tangent bundle $\mathrm{T}\cM$ determined by a collection of $C^\infty$-smooth vector fields $X_1, \ldots, X_d$ via $\mathrm{H}_p \cM = \mathrm{span} (X_1|_p , \ldots, X_d|_p)$. Since $\cM$ is connected and $\rH \cM$ is a vector bundle, $\dim(\rH_p \cM)$ is independent of $p$ and called the rank of the distribution.

Suppose $\cM$ is a manifold equipped with a distribution $\mathrm{H}\cM$ as above. Let $\Upomega \subset \cM$ be open. We write $\fv^k (\Upomega)$ for the $C^k$-smooth sections of $\mathrm{T}\Upomega$ and $\fv_{\mathrm{H}}^k (\Upomega)$ for the $C^k$-smooth sections of $\mathrm{H}\Upomega$. Set $\fv_1 = \fv_{\mathrm{H}}^\infty (\cM)$, and define by iteration $\fv_k = \fv_{k-1} \oplus [\fv_1 , \fv_{k-1}]$. Let $\Gamma^k$ be determined by $\Gamma_p^k = \mathrm{span}\{ V|_p \,:\, V \in \fv_k \}$. The distribution $\mathrm{H}\cM$ is called regular if $\Gamma^k$ is a distribution for all $k$.

When $\mathrm{H}\cM$ is regular, there exists a unique positive integer $s$ such that
\[
	\mathrm{H}\cM = \Gamma^1 \subsetneq \Gamma^2 \subsetneq \cdots \subsetneq \Gamma^s
\]	
and $\Gamma^k = \Gamma ^s$ for all $k > s$. If we set $\fl_1 (p) = \Gamma_p^1$, and for $k\geq 2$, $\fl_k (p) = \Gamma_p^k / \Gamma_p^{k-1}$, then $\fl(p) = \bigoplus_{k\geq 1}\fl_k (p)$ can be given a bracket as in \cite[p.~9]{Tanaka1} making it a stratified Lie algebra of step $s$. See Section \ref{sec:basics} for the definition of a stratified Lie algebra (of step $s$) and the early sections of \cite{Tanaka1}, \cite{Yamaguchi} for more details on the objects of this paragraph. If there exists a stratified Lie algebra $\fl$ such that $\fl(p)$ is isomorphic to $\fl$ at every $p \in \cM$ then the pair $(\cM,\mathrm{H}\cM)$ is called of type $\fl$.  Somewhat conversely, if $\fL$ is a Lie group such that its Lie algebra $\fl = \bigoplus_{k\geq 1}\fl_k$ is stratified of step $s$, then $\fl_1$ determines a regular distribution $\mathrm{H}\fL$ on $\fL$ and $(\fL, \mathrm{H}\fL)$ is of type $\fl$.      

Let $\cM, \cM'$ be $C^\infty$-smooth manifolds and let $\Upomega \subset \cM$ be open. We call $f: \Upomega \to \cM'$ a $C^k$-diffeomorphism if $f : \Upomega \to f(\Upomega)$ is a homeomorphism and both $f$ and $f^{-1}$ are $C^k$-smooth. When $\cM, \cM'$ are connected and $\mathrm{H}\cM, \mathrm{H}\cM'$ are distributions on $\cM, \cM'$, respectively, we call $f: \Upomega \to \cM'$ a ($C^k$-smooth) contact mapping if it is a $C^k$-diffeomorphism for some $k\geq 1$ and $Df (\mathrm{H}\Upomega) = \mathrm{H}f(\Upomega)$. This implies $f_* \fv_{\mathrm{H}}^\infty (\Upomega) \subset \fv_{\mathrm{H}}^\infty \left( f(\Upomega)\right)$. If $f : \Upomega \to \cM'$ is a contact mapping then so too is $f^{-1} : f(\Upomega) \to \cM$. The pairs $(\cM,\mathrm{H}\cM)$ and $(\cM',\mathrm{H}\cM')$ are called locally equivalent if for all $p \in \cM$, and for all $q \in \cM'$, there exist open $\Upomega \subset \cM$ with $p \in \Upomega$, open $\Upomega' \subset \cM'$ with $q \in \Upomega'$, and a $C^\infty$-smooth contact mapping $f : \Upomega \to \Upomega'$.

The pair $(\cM,\mathrm{H}\cM)$ is called a $(2,3,5)$-distribution if $\cM$ is a connected $C^\infty$-smooth manifold with $\dim(\cM)=5$ and $\mathrm{H}\cM$ is a regular distribution on $\cM$ satisfying $\mathrm{rank}(\mathrm{H}\cM) = \mathrm{rank}(\Gamma^1) = 2$, $\mathrm{rank}(\Gamma^2) = 3$, and $\mathrm{rank}(\Gamma^3) = 5$. These were studied by Cartan in his famous ``five variables'' paper \cite{CartanFiveVariables}. This work claimed a classification long accepted as complete until a missing case was discovered by Doubrov and Gogorov in 2013 \cite{MissingFromCartan}. See also the commentary in \cite{Willse}. Cartan discussed $(2,3,5)$-distributions in the context of space curves of constant torsion. They have since been linked to the symmetries of rolling balls \cite{RollingBalls} or, if you prefer, a rolling spinor \cite{RollingSpinor}.

A $(2,3,5)$-distribution is called flat if a certain invariant (quantity preserved under local equivalence) vanishes. Every flat $(2,3,5)$-distribution is locally equivalent to the Lie group we study in Section \ref{sec:235} that is commonly called the Cartan group (and we follow suit). The Cartan group $\fC$ is a rigid stratified Lie group (definitions in Section \ref{sec:basics}). The main purpose of this paper is to prove the following.

\begin{thm}\label{thm:main}
Let $\Upomega \subset \fC$ be open. If $f : \Upomega \to \fC$ is a $C^1$-smooth contact mapping then $f$ is $C^\infty$-smooth.
\end{thm}

\noindent This is achieved in Theorem \ref{thm:end} of Section \ref{sec:235}. Our paper can be thought of as a detailed case study of a natural extension to methods making use of the Tanaka prolongation of a stratified Lie algebra. We hope it may provide insight into the way the structure of a rigid group controls the geometry and function theory of that group.

In \cite{OttazziWarhurstGeneral} it was shown that if $\fL$ is any rigid stratified Lie group, $\Upomega \subset \fL$ is open, and $f : \Upomega \to \fL$ is a $C^2$-smooth contact mapping then $f$ is $C^\infty$-smooth. We are improving on this result in a special case.

The result of \cite{OttazziWarhurstGeneral} just referred to is achieved as follows. If $V \in \fv^1 (\Upomega)$ is contact, then it may be regularized so that the smooth approximations are also contact. A $C^\infty$-smooth contact vector field on a domain of a rigid stratified Lie group is known to have polynomial coefficients, with the degree of these polynomials dependent only on the group structure. Indeed the $C^\infty$-smooth contact vector fields on a domain of a rigid group form a finite-dimensional Lie algebra. Consequently, the approximating vector fields limit on the original in a finite-dimensional subspace. Thus the original $C^1$-smooth contact vector field has polynomial coefficients and in fact $V \in \fv^\infty (\Upomega)$. This is utilized by taking the vector field generator $V$ of a family of $C^\infty$-smooth contact mappings, left-translations say, and considering $f_* V$ with $f$ a $C^2$-smooth contact mapping on $\Upomega$. Since
\begin{equation}\label{eq:c2contact}
    [f_* V , \fv_{\rH}^\infty (f(\Upomega))] \subset f_* [V, \fv_{\rH}^\infty (\Upomega)] \subset v_{\rH}^\infty (f(\Upomega))  
\end{equation}
we have by definition that $f_* V$ is a $C^1$-smooth contact vector field on $f(\Upomega)$. The previous argument implies $f_* V$ is $C^\infty$-smooth and this can be used to deduce that $f$ itself is smooth. When $f$ is assumed only $C^1$-smooth, \eqref{eq:c2contact} is not justified. The computations of Section \ref{sec:235} are a means of circumnavigating this obstacle.

Let us take a closer look at the guiding principle behind those computations. Suppose $\fL$ is a stratified Lie group with $\fl = \mathrm{Lie}(\fL)$. If $X \in \fl$ then $p \mapsto \exp (t X) p$ is left-translation by $\exp (t X)$, a family of $C^\infty$-smooth contact mappings indexed by parameter $t$. Let $V|_p = \frac{\rd}{\rd t} \big|_{t=0} \exp (t X) p$. With $f$ a contact mapping on a domain $\Upomega \subset \fL$ we consider
\[
    h_t (q) = f \left(\exp (tX) f^{-1} (q)\right)
\]
which is a family of contact mappings (on $f(\Upomega)$) satisfying $h_0 (q) = q$. Assuming everything is smooth, we notice that $\dot{h}_0 (q) = (f_* V)|_q = (\widetilde{X} f)|_{f^{-1}(q)}$ with $\widetilde{X}$ the right-invariant mirror of $X$. Of course, if $X$ is in the center of $\fl$ then $\widetilde{X} = X$. We are able to rely on such vector fields in Section \ref{sec:235}. 

Suppose $\fL$ is a rigid group and $\fv_{\mathrm{C}}(\Upomega)$ is the finite-dimensional Lie algebra of $C^\infty$-smooth contact vector fields on $\Upomega$ (isomorphic to a Lie algebra independent of $\Upomega$ we will denote $\fv_{\mathrm{C}}$). Once a contact mapping $f$ on $\Upomega$ has been shown to be $C^\infty$-smooth it is found to be induced by an element of the automorphism group of $\fv_{\mathrm{C}}$. See the paragraphs at the end of Section 3 in \cite{OttazziWarhurstGeneral}. In the case of the Cartan group, $\fv_{\mathrm{C}}$ is isomorphic to $\fg_2$ the real ``split form'' of the $14$-dimensional exceptional complex simple Lie algebra $\fg_2^\mC$. For a painstaking calculation of a basis of vector fields for $\fv_{\mathrm{C}}(\fC)$ see \cite{Sachkov}.

A $C^1$-smooth assumption could be said to be natural when defining a contact mapping, however, it leaves something to be desired. Better would be a horizontal Sobolev-type condition. That is, if $\bigoplus_{i\geq 1}\fl_i = \mathrm{Lie}(\fL)$ and $X_1,\ldots,X_d$ is a basis for $\fl_1$ then we would like to assume only that $f : \Upomega \to f(\Upomega)$ is a homeomorphism and that the components of the distributional derivatives $X_j f$ are in $L_{\mathrm{loc}}^r (\Upomega)$ for some $r \geq 1$. This setup can then be used to define contact in a weak sense. It could be argued that exactly what $r$ should be is yet to be decided. It seems quite likely that, with some algebraic dexterity, our approach could give a $C^1$-smooth result for all rigid groups. Whether a variation on these methods has application in the Sobolev setting remains to be seen, but there is reason to believe it would require some delicate analysis. This is in large part due to the appearance of right-invariant derivatives as mentioned above.   

A horizontal Sobolev-type condition is the right assumption in the study of (locally) quasiconformal mappings (for equivalence of the analytic definition with other definitions). We do not pursue this topic in any detail here for the sake of brevity. We should, however, make note of recent work on Xie's conjecture. This conjecture states that if $\fL$ is a stratified Lie group (equipped with its canonical Carnot-Carth\'{e}odory distance function) other than $\mR^n$ or $\mH^n$ (the $n$-th Heisenberg group) then every quasiconformal mapping $f : \Upomega \to \fL$ is locally bi-Lipschitz, and moreover if $\Upomega = \fL$ then $f$ is bi-Lipschitz. The recent paper \cite{KMX} proves this in the case of non-rigid Carnot groups. In the case of rigid groups, the conjecture would be verified if it were true that all quasiconformal mappings were $C^\infty$-smooth (once coupled with the results of \cite{CowOttGlobal}). All geometric mappings such as bi-Lipschitz and quasiconformal mappings must be weakly contact. On the other hand, a $C^1$-smooth contact mapping is necessarily locally quasiconformal. In \cite{KMX} the authors mention that rigid groups will be discussed in a forthcoming article.

If $\fL$ is an $H$-type group with center of dimension at least $3$ (these groups are rigid), then a quasiconformal mapping $f : \Upomega \to \fL$ is $C^\infty$-smooth by the regularity results of \cite{CapognaRegularity99} and \cite{Capogna&Cowling}. These rely on non-linear potential theory. We emphasize that our results are achieved using linear operators only. Our program of using vector flow methods in conjunction with linear hypoelliptic operators was begun in collaboration with Jeremy Tyson in \cite{AustinTyson}, and was to some extent influenced by \cite{LiuAnother} and \cite{Sarvas}. These papers are all related to conformal, or $1$-quasiconformal mappings and Liouville-type theorems for them.  

\subsection{Acknowledgements}
The author thanks Alessandro Ottazzi for comments on the Tanaka prolongation of a stratified Lie algebra, Francesco Serra Cassano for elaborating on the proof of Lemma \ref{lem:scc1}, and Ben Warhurst for suggesting the Cartan group as suitable test case.

\section{Contact vector fields on a stratified Lie group}\label{sec:basics}

A Lie algebra $\fl$ is called stratified (of step $s$) if $\dim(\fl)<\infty$ and there are vector spaces $\fl_i$ such that $\fl = \bigoplus_{i\geq 1}\fl_i$, $[\fl_{1} , \fl_i] = \fl_{i+1}$ for all $i \geq 1$, $\fl_{s} \neq \{0\}$, and $\fl_i = \{0\}$ for all $i > s$. A Lie group $\fL$ is called stratified if it is connected, simply connected, and $\fl = \mathrm{Lie}(\fL)$ is stratified. Nothing essential is lost if a stratified Lie group is regarded as $\mR^n$ with a polynomial group law. Indeed, if $\fL$ is a stratified Lie group and $\Upomega \subset \fL$ is open then we (tacitly) identify $\Upomega$ with an open subset of $\mR^n$ whenever it is convenient to do so.

If $\fL$ is a stratified Lie group and $X_1,\ldots, X_d$ is a basis for $\fl_{1}$ then $\mathrm{H}\fL$ denotes the distribution determined by $\mathrm{H}_p \fL = \mathrm{span} (X_1|_p,\ldots,X_d|_p)$. If $\Upomega \subset \fL$ is open then $\rH \Upomega$ is the induced distribution and we write $\fv_{\mathrm{H}}^k (\Upomega)$ for the $C^k$-smooth sections of $\mathrm{H}\Upomega$. Elements of $\fv_{\mathrm{H}}^k (\Upomega)$ will be referred to as horizontal vector fields. When $k\geq 1$, a vector field $V \in \fv^k (\Upomega)$ is called contact if $[V,\fv_{\mathrm{H}}^\infty (\Upomega)] \subset \fv_{\mathrm{H}}^\infty (\Upomega)$.

A simple prolongation of a stratified Lie algebra $\fl = \bigoplus_{i\geq 1}\fl_{i}$ is a graded Lie algebra $\fh = \bigoplus_{i \in \mZ} \fh_i$ satisfying
\begin{enumerate}[(i)]
	\item for all $i \leq -1$, $\fh_i = \fl_{-i}$ and
	\item for all $i \geq 0$, if $Z \in \fh_i$ is such that $[Z , \fh_{-1}] = \{0\}$ then $Z = 0$.
\end{enumerate}
The Tanaka prolongation of $\fl$ is the simple prolongation $\ft (\fl) = \bigoplus_{i\in \mZ} \ft_i (\fl)$ such that whenever $\fh = \bigoplus_{i\in \mZ}\fh_i$ is another simple prolongation, there is an injective Lie algebra homomorphism $\xi : \fh \to \ft (\fl)$ with $\xi (\fh_i) \subset \ft_i (\fl)$ for all $i$. See \cite[pp.~23-25]{Tanaka1} for the construction of $\ft (\fl)$. A stratified Lie group $\fL$ with Lie algebra $\fl$ is called rigid if there exists $N \geq 0$ such that $\ft_i (\fl) = \{0\}$ for all $i \geq N$. The following can be found near the end of Section 3 in \cite{OttazziWarhurstGeneral}.

\begin{pro}
	Let $\fL$ be a stratified Lie group of dimension $n$ and let $\Upomega \subset \fL$ be open. If $\fL$ is rigid then there is $N \geq 0$ with the following property: if $Y_1, \ldots, Y_n$ is a basis for $\fl = \mathrm{Lie}(\fL)$ and $V = \sum_{i=1}^n v_i Y_i \in \fv^1 (\Upomega)$ is a contact vector field, then the component functions $v_i$ are polynomials of degree no greater than $N$.
\end{pro} 

Should $\fv^1 (\Upomega)$ be replaced by $\fv^\infty (\Upomega)$ (see the second paragraph of the introduction for notation) then the statement follows almost immediately from the constructions of Section 2.6 in \cite{OttazziWarhurstGeneral}. Those constructions took inspiration from Section 2 of \cite{Yamaguchi}, which is largely an exposition of Section 6 in \cite{Tanaka1}. As stated, the proof of the proposition uses a smoothing argument which had appeared before in \cite[p. 83]{CDMKR}. 

When $\fL$ is a stratified Lie group and $\Upomega \subset \fL$ is open, we write $\cD'(\Upomega)$ for the real-valued distributions on $\Upomega$. Since there is an overlap of vocabulary we emphasize that \textit{distributions} is being used here in the sense of generalized functions. They are continuous linear functionals on $(C_0^{\infty}(\Upomega), \tau)$ with $\tau$ the usual topology. Convergence in $\cD' (\Upomega)$ is with reference to the weak-$\star$ topology. With $\fl = \mathrm{Lie}(\fL)$, define $\fv^{< 1} (\Upomega) = \cD' (\Upomega) \otimes \fl$. We think of $\fv^{<1} (\Upomega)$ as the space of finite formal sums
\[
    \left\{ \sum \alpha_i Z_i \;\bigg|\; \alpha_i \in \cD' (\Upomega) , \, Z_i \in \fl \right\}
\]
(which it is, modulo the null sums). We sometimes call an element of $\fv^{<1} (\Upomega)$ a \textbf{generalized vector field} (on $\Upomega$).

Suppose $\fL$ is a stratified Lie group with $\fl = \mathrm{Lie}(\fL)$ of step $s$. Let $\{Y_{i,j}\}$ be a basis for $\fl$ such that $Y_{i,1},\ldots,Y_{i,d_i}$ is a basis for $\fl_i$. For all $i =2\ldots,s$, $j = 1,\ldots, d_i$, and $k = 1,\ldots,d_1$ there is a linear combination $Q_{i,j,k}$ acting on $d_{i-1}$ objects such that for all open $\Upomega\subset \fL$ and for all $\alpha_1,\ldots, \alpha_{d_{i-1}} \in \cD' (\Upomega)$ we have 
\[
	\sum_{j=1}^{d_{i-1}} \alpha_j [Y_{i-1,j}, X_k] = \sum_{j=1}^{d_i} Q_{i,j,k} (\alpha_1,\ldots, \alpha_{d_{i-1}}) Y_{i,j} 
\]
in $\fv^{<1}(\Upomega)$ for all $i=2,\ldots,s$ and $k=1,\ldots,d_1$. A generalized vector field $\Upsilon = \sum_{i=1}^s \sum_{j=1}^{d_i} \alpha_{i,j} Y_{i,j}$ is called \textbf{contact} if for all $i = 2, \ldots, s$, $j = 1, \ldots, d_i$, and $k = 1, \ldots, d_1$ we have
\[
	X_k \alpha_{i,j} = Q_{i,j,k}(\alpha_{i-1,1},\ldots,\alpha_{i-1,d_{i-1}}).
\]

If $\fL$ is a stratified Lie group, $\Upomega \subset \fL$ is open, and $\Upomega_0 \subset 
\Upomega$ is also open then $\alpha \in \cD' (\Upomega)$ determines $\alpha_0 \in \cD' (\Upomega_0)$ via $\left<\alpha_0 , \phi\right> = \left<\alpha , \phi\right>$ for all $\phi \in C_0^\infty (\Upomega_0)\subset C_0^\infty (\Upomega)$. If $\overline{\Upomega}_0 \subset \Upomega$ and $\overline{\Upomega}_0$ is compact (from now on denoted $\Upomega_0 \subset\subset \Upomega$) then there exists $\epsilon_0 >0$ with the following property: whenever $\alpha \in \cD'(\Upomega)$ there is a family $\{ \alpha^{\epsilon} \,|\, 0<\epsilon <\epsilon_0\}$ of $C^\infty$-smooth functions defined on $\Upomega_0$ such that $\alpha^{\epsilon} \to \alpha_0$ in $\cD' (\Upomega_0)$ as $\epsilon \to 0$. These smooth approximations to $\alpha_0$ are achieved by convolving $\alpha$ with suitable elements of $C_0^\infty (\fL)$. For an exposition of this theory see \cite[pp.~88-90]{HormanderI}. The manner of the regularization implies that $Z\alpha^\epsilon = (Z\alpha)^\epsilon$ for all $Z \in \fl$. Moreover, if $\alpha, \beta \in \cD'(\Upomega)$ then $(s\alpha + t\beta)^{\epsilon} = s\alpha^{\epsilon} + t\beta^{\epsilon}$ in $\cD'(\Upomega_0)$ for all $s,t \in \mR$. 

Suppose $\alpha^\epsilon \to \alpha_0$ in $\cD' (\Upomega_0)$ as in the previous paragraph and that each $\alpha^\epsilon$ is a polynomial of degree at most $N$ (with $N$ independent of $\epsilon$). The distributions that may be identified with polynomials of degree at most $N$ form a finite-dimensional subspace of $\cD' (\Upomega_0)$. Since $\cD' (\Upomega_0)$ is a (locally convex and) Hausdorff topological vector space, this subspace is closed. It follows that $\alpha_0$ may also be identified with a polynomial of degree at most $N$.

Every open set $\Upomega \subset \fL$ admits a compact exhaustion. That is, there is a sequence of open sets $(\Upomega_k)_{k=1}^\infty$ such that $\Upomega_k \subset\subset \Upomega$, $\Upomega_k \subset \Upomega_{k+1}$, and $\Upomega = \bigcup \Upomega_k$. If $\alpha_k$ is the restriction of $\alpha$ to $C_0^\infty (\Upomega_k)$ (as discussed in the case of $k=0$), and each $\alpha_k$ is found to be a polynomial, then the polynomials are the same and $\alpha$ may be identified with a polynomial on $\Upomega$. If $\alpha$ is a continuous function to begin with, and $\alpha$ as distribution may be identified with a polynomial, then $\alpha$ is that polynomial.  

\begin{pro} \label{pro:contactvfpoly}
Let $\fL$ be a stratified Lie group with Lie algebra $\fl = \bigoplus_{i\geq 1} \fl_i$. Let $d_i = \dim(\fl_i)$ and let $Y_{i,j}$ with $j=1,\ldots, d_i$ be a basis for $\fl_i$. Let $\Upomega \subset \fL$ be open and suppose $\Upsilon = \sum_i\sum_j \alpha_{i,j} Y_{i,j}$ is a contact generalized vector field on $\Upomega$. If $\fL$ is rigid then each $\alpha_{i,j}$ may be identified with a polynomial on $\Upomega$.
\end{pro}

\begin{proof}
Let $\Upomega_0$ be open such that $\Upomega_0 \subset\subset \Upomega$. Let $\alpha_{i,j}^{\epsilon}$ be a sequence of $C^\infty$-smooth functions defined on $\Upomega_0$ as discussed in the paragraphs preceding this proof. It follows $\Upsilon^{\epsilon} = \sum_i \sum_j \alpha_{i,j}^{\epsilon} Y_{i,j}$ is a $C^\infty$-smooth vector field on $\Upomega_0$. Let $X_j = Y_{1,j}$ and suppose $V \in \fv_{\mathrm{H}}^\infty (\Upomega_0)$. It follows $V =\sum_j \sigma_j X_j$ with each $\sigma_j \in C^\infty (\Upomega_0)$. Simple rearrangements and a replacement justified by the definition of $Q_{i,j,k}$ alone yield
\begin{align*}
	[\Upsilon^{\epsilon}, V] &= \sum_{k=1}^{d_1}\sum_{i=1}^{s}\sum_{j=1}^{d_i} [\alpha_{i,j}^{\epsilon} Y_{i,j}, \sigma_k X_k] \\
&= \sum_{k=1}^{d_1}\sum_{i=1}^{s}\sum_{j=1}^{d_i} \left(\alpha_{i,j}^{\epsilon} \sigma_k [Y_{i,j}, X_k] + \alpha_{i,j}^{\epsilon} (Y_{i,j} \sigma_k) X_k  - \sigma_k (X_k \alpha_{i,j}^{\epsilon})Y_{i,j}\right)\\
&= \sum_{k=1}^{d_1}\sum_{i=1}^{s}\sum_{j=1}^{d_i} \alpha_{i,j}^{\epsilon} (Y_{i,j} \sigma_k) X_k - \sum_{k=1}^{d_1}\sum_{j=1}^{d_1} \sigma_k (X_k \alpha_{1,j}^{\epsilon})X_j \\
& \quad + \sum_{k=1}^{d_1} \sigma_k \left( \sum_{i=2}^{s+1} \sum_{j=1}^{d_{i-1}} \alpha_{i-1,j}^{\epsilon} [ Y_{i-1,j}, X_k] - \sum_{i=2}^{s} \sum_{j=1}^{d_i} (X_k \alpha_{i,j}^{\epsilon})Y_{i,j} \right) \\
&=\sum_{k=1}^{d_1}\sum_{i=1}^{s}\sum_{j=1}^{d_i} \alpha_{i,j}^{\epsilon}(Y_{i,j} \sigma_k) X_k - \sum_{k=1}^{d_1}\sum_{j=1}^{d_1} \sigma_k (X_k \alpha_{1,j}^{\epsilon})X_j \\
& \quad + \sum_{k=1}^{d_1} \sigma_k \sum_{i=2}^{s} \sum_{j=1}^{d_{i}} \left(Q_{i,j,k}(\alpha_{i-1,1}^{\epsilon},\ldots,\alpha_{i-1,d_{i-1}}^{\epsilon}) - X_k \alpha_{i,j}^{\epsilon}\right)Y_{i,j}.
\end{align*} 
Since
\[
	X_k \alpha_{i,j}^{\epsilon} = (X_k \alpha_{i,j})^{\epsilon} = Q_{i,j,k}(\alpha_{i-1,1},\ldots,\alpha_{i-1,d_{i-1}})^{\epsilon} = Q_{i,j,k}(\alpha_{i-1,1}^{\epsilon},\ldots,\alpha_{i-1,d_{i-1}}^{\epsilon})
\]
we have that
\[
	[\Upsilon^{\epsilon}, V] = \sum_{k=1}^{d_1}\sum_{i=1}^{s}\sum_{j=1}^{d_i} \alpha_{i,j}^{\epsilon} (Y_{i,j} \sigma_k) X_k - \sum_{k=1}^{d_1}\sum_{j=1}^{d_1} \sigma_k (X_k \alpha_{1,j}^{\epsilon})X_j.
\]
This is clearly a horizontal vector field on $\Upomega_0$, hence $\Upsilon^{\epsilon}$ is contact in the classical sense. Consequently, there is $N$ such that $\alpha_{i,j}^{\epsilon}$ is a polynomial of degree at most $N$ for all $i$, $j$, and suitable $\epsilon$. Our comments above imply that each $\alpha_{i,j}$ may be identified with a polynomial on $\Upomega$.
\end{proof}

\section{Contact mappings on the Cartan group}\label{sec:235}
The set $\mR^5$ with group product 
\[
    (x_1,x_2,y,z_1,z_2)(x_1',x_2', y',z_1',z_2') = (P_1,P_2,P_3,P_4,P_5),
\]
\begin{align*}
    P_1 &= x_1+x_1', \\
    P_2 &= x_2+x_2', \\
    P_3 &= y+y' + \tfrac{1}{2}(x_1 x_2' - x_2 x_1'), \\
    P_4 &= z_1 + z_1' + \tfrac{1}{2}(x_1 y' - y x_1') +\tfrac{1}{12}\left( (x_1 - x_1' ) (x_1 x_2' - x_2 x_1') \right), \text{ and} \\
    P_5 &= z_2 + z_2' + \tfrac{1}{2}(x_2 y' - y x_2') +\tfrac{1}{12}\left( (x_2 - x_2' ) (x_1 x_2' - x_2 x_1') \right),
\end{align*}
is a connected, simply connected Lie group we denote $\fC$. It is a realization of what is sometimes called the Cartan group. We choose the following basis for $\fc = \mathrm{Lie}(\fC)$: 
\begin{align*}
    X_1 &= \partial_{x_1} - \tfrac{1}{2}x_2 \partial_y -\tfrac{1}{2}(y+\tfrac{1}{6}x_1 x_2)\partial_{z_1} - \tfrac{1}{12}x_2^2\partial_{z_2}, \\
    X_2 &= \partial_{x_2} + \tfrac{1}{2}x_1 \partial_y + \tfrac{1}{12}x_1^2\partial_{z_1} - \tfrac{1}{2}(y-\tfrac{1}{6}x_1 x_2)\partial_{z_2}, \\
    Y &= \partial_y + \tfrac{1}{2}x_1 \partial_{z_1} + \tfrac{1}{2} x_2 \partial_{z_2}, \\
    Z _1 &= \partial_{z_1}, \text{ and} \\
    Z _1 &= \partial_{z_2}.
\end{align*}
On occasion it will be convenient to have the alternative labels $Y_1 = X_1$, $Y_2 = X_2$, $Y_3 = Y$, $Y_4 = Z_1$, and $Y_5 = Z_2$. The explicit expressions for these vector fields are achieved by $Y_j|_p = D_0 L_p (\partial_j)$, with $\partial_1, \ldots, \partial_5$ the canonical Euclidean basis at the origin of $\mR^5$.

It is easily found that
\[
    [X_1, X_2] = Y, \quad [X_1, Y] = Z_1,\quad \text{and} \quad [X_2, Y] = Z_2.
\]
All brackets not immediate consequences of these are trivial. Consequently, $\fc$ is a stratified Lie algebra of step $3$ with 
\[
    \fc_{1} = \text{span}( X_1, X_2), \quad \fc_{2} = \text{span}( Y ), \quad\text{and}\quad \fc_{3} = \text{span}(Z_1, Z_2)
\]
and $\fc_1$ determines a $(2,3,5)$-distribution. 

\begin{rem}
The group product is that arising from the following procedure (in brief). Begin with an abstract Lie algebra $\tilde\fc$ with basis $\widetilde{X}_1,\ldots,\widetilde{Z}_2$ satisfying the relations just discussed ($\widetilde{X}_1$ replaces $X_1$ etc.). There is a connected, simply connected abstract Lie group $\widetilde\fC$ with $\mathrm{Lie}(\widetilde\fC)=\tilde\fc$ and such that $\exp : \tilde\fc \to \widetilde\fC$ is a $C^\infty$-smooth diffeomorphism. Now identify $(x_1,x_2,y,z_1,z_2)\in \mR^5$ with $\exp (x_1 \widetilde{X}_1 + x_2 \widetilde{X}_2 + y \widetilde{Y} + z_1 \widetilde{Z}_1 + z_2 \widetilde{Z}_2)$ and use the Baker-Campbell-Hausdorff formula to discover a group law making $\mR^5$ equipped with that law isomorphic to $\widetilde\fC$.
\end{rem}

It is well known that $\fC$ is rigid, indeed $\ft(\fc)$ is isomorphic to the exceptional simple Lie algebra of dimension $14$. More details can be found in \cite[pp.~29-30]{Tanaka1}. 

A basis dual to $X_1, X_2, Y, Z_1, Z_2$ for the $1$-forms on $\fC$ is given by
\begin{align*}
	\eta_1 &= \mathrm{d}x_1, \\
	\eta_2 &= \mathrm{d}x_2, \\
	\theta &= \mathrm{d}y + \tfrac{1}{2}x_2 \mathrm{d}x_1 - \tfrac{1}{2} x_1 \mathrm{d}x_2, \\
	\iota_1 &= \mathrm{d}z_1 + \left(\tfrac{1}{2} y - \tfrac{1}{6}x_1 x_2 \right) \mathrm{d}x_1 + \tfrac{1}{6} x_1^2 \mathrm{d}x_2 + \tfrac{1}{2} x_1 \mathrm{d}y, \text{ and}\\
\iota_2 &= \mathrm{d}z_2 -\tfrac{1}{6} x_2^2 \mathrm{d}x_1 +  \left(\tfrac{1}{2} y + \tfrac{1}{6}x_1 x_2 \right)\mathrm{d}x_2 - \tfrac{1}{2} x_2 \mathrm{d}y.
\end{align*}
We will sometimes refer to these by $\theta_1 = \eta_1$, $\theta_2 = \eta_2$, $\theta_3 = \theta$, $\theta_4 = \iota_1$, and $\theta_5 = \iota_2$.

Like any stratified Lie group, $\fC$ admits a family $\{\delta_r \,|\, r \in (0,\infty)\}$ of group automorphisms called homogeneous dilations,
\begin{equation*}
    \delta_r (x_1,x_2,y,z_1,z_2) = (r x_1, r x_2, r^2 y, r^3 z_1, r^3 z_2).
\end{equation*}

If $\Upomega \subset \fC$ is open and $f : \Upomega \to \fC$ is a $C^1$-smooth contact mapping, we define
\[
	J_H f = \det \begin{pmatrix} X_1 f_1 & X_2 f_1 \\ X_1 f_2 & X_2 f_2 \end{pmatrix}.
\]
It is sometimes called the horizontal Jacobian of $f$.

We rely on the following result of Warhurst from \cite{WarhurstC1}.
\begin{thm} \label{thm:c1pansu}
Let $\fL$ be a stratified Lie group, let $\Upomega \subset \fL$ be open, and let $f : \Upomega \to \fL$ be a $C^1$-diffeomorphism. Then $f$ is a contact mapping if and only if $f$ is Pansu-differentiable.
\end{thm}

Let $\Upomega \subset \fC$ be an open subset and let $f : \Upomega \to \fC$ be a $C^1$-smooth contact mapping. Let $\Upomega' = f(\Upomega)$ and define $g : \Upomega' \to \Upomega$ by $g = f^{-1}$. These are fixed for the remainder of the section.  

By Theorem \ref{thm:c1pansu} $f$ is Pansu-differentiable. This means for each $p \in \Upomega$ there is a graded Lie algebra automorphism $\cP_p f$ whose action on $\mathrm{T}_0\fC$ is given by
\[
	\cP_p f ( X|_0 ) = \lim_{r\to 0} \exp^{-1}\left(\delta_{\frac{1}{r}} \left( f(p)^{-1} f(p\delta_r ( \exp X|_0 ))\right)\right).
\]

Using only the structural properties of $\fc$ we find
\begin{align*}
    \cP f (Y|_0) &= [\cP f (X_1|_0) , \cP f (X_2|_0)] = (J_H f) Y|_0, \\
    \cP f (Z_1|_0) &= J_H f \left(\left(X_1 f_1\right) Z_1|_0 + \left(X_1 f_2\right) Z_2|_0\right), \text{ and} \\
    \cP f (Z_2|_0) &= J_H f \left(\left(X_2 f_1\right) Z_1|_0 + \left(X_2 f_2\right) Z_2|_0\right).
\end{align*}
Combining these with evaluation of the limit as $r \to 0$ of the component functions of $\exp^{-1}\left(\delta_{1/r} \left( f(p)^{-1} f(p\delta_r (\exp X|_0)) \right)\right)$ for different choices of $X|_0 \in \mathrm{T}_0 \fC$ we discover the following identities:
\begin{align*}
    (\cP f)_{3,1} &= X_1 f_3 + \tfrac{1}{2} (f_2 X_1 f_1  - f_1  X_1 f_2 ) = 0, \\
    (\cP f)_{4,1} &= X_1 f_4 + \tfrac{1}{2} ( f_3 X_1 f_1 - f_1 X_1 f_3  ) - \tfrac{1}{6}(f_1 (f_2 X_1 f_1 - f_1 X_1 f_2)) = 0, \\
    (\cP f)_{5,1} &= X_1 f_5 + \tfrac{1}{2} ( f_3 X_1 f_2 - f_2 X_1 f_3  ) - \tfrac{1}{6}(f_2 (f_2 X_1 f_1 - f_1 X_1 f_2)) = 0, \\
    (\cP f)_{3,2} &= X_2 f_3 + \tfrac{1}{2} (f_2 X_2 f_1  - f_1  X_2 f_2 ) = 0, \\
    (\cP f)_{4,2} &= X_2 f_4 + \tfrac{1}{2} ( f_3 X_2 f_1 - f_1 X_2 f_3  ) - \tfrac{1}{6}(f_1 (f_2 X_2 f_1 - f_1 X_2 f_2)) = 0, \\
    (\cP f)_{5,2} &= X_2 f_5 + \tfrac{1}{2} ( f_3 X_2 f_2 - f_2 X_2 f_3  ) - \tfrac{1}{6}(f_2 (f_2 X_2 f_1 - f_1 X_2 f_2)) = 0, \\
    (\cP f)_{3,3} &= Y f_3 + \tfrac{1}{2} (f_2 Y f_1  - f_1  Y f_2 ) = J_H f, \\
    (\cP f)_{4,3} &= Y f_4 + \tfrac{1}{2} ( f_3 Y f_1 - f_1 Y f_3  ) - \tfrac{1}{6}(f_1 (f_2 Y f_1 - f_1 Y f_2)) = 0, \\
    (\cP f)_{5,3} &= Y f_5 + \tfrac{1}{2} ( f_3 Y f_2 - f_2 Y f_3  ) - \tfrac{1}{6}(f_2 (f_2 Y f_1 - f_1 Y f_2)) = 0, \\ 
    (\cP f)_{4,4} &= Z _1 f_4 + \tfrac{1}{2} ( f_3 Z _1 f_1 - f_1 Z _1 f_3  ) - \tfrac{1}{6}(f_1 (f_2 Z _1 f_1 - f_1 Z _1 f_2)) = J_H f (X_1 f_1), \\
    (\cP f)_{5,4} &= Z _1 f_5 + \tfrac{1}{2} ( f_3 Z _1 f_2 - f_2 Z _1 f_3  ) - \tfrac{1}{6}(f_2 (f_2 Z _1 f_1 - f_1 Z _1 f_2)) = J_H f (X_1 f_2), \\
    (\cP f)_{4,5} &= Z _2 f_4 + \tfrac{1}{2} ( f_3 Z _2 f_1 - f_1 Z _2 f_3  ) - \tfrac{1}{6}(f_1 (f_2 Z _2 f_1 - f_1 Z _2 f_2)) = J_H f (X_2 f_1), \text{ and} \\
    (\cP f)_{5,5} &= Z _2 f_5 + \tfrac{1}{2} ( f_3 Z _2 f_2 - f_2 Z _2 f_3  ) - \tfrac{1}{6}(f_2 (f_2 Z _2 f_1 - f_1 Z _2 f_2)) = J_H f (X_2 f_2).
\end{align*}
Here we have written $(\cP f)_{i,j}$ for the $(i,j)$-entry in the matrix $(\cP f)$ of $\cP f$ with respect to the basis $X_1|_0, X_2|_0, Y|_0, Z_1|_0, Z_2|_0$. The matrix $(\cP f)$ is of block-diagonal form and it is easily found that $\det (\cP f) = J_H ^5 f$. Though we do not use the observation, it may be worth noting that for the combinations of $i,j$ appearing in the above list, $(\cP f)_{i,j} = \left< f^* \theta_i , Y_j \right>$.

Before proceeding further, we require some notation: if $\fL$ is a stratified Lie group with $\mathrm{Lie}(\fl)=\bigoplus_{i\geq 1}\fl_i$ and $\Lambda \subset \fL$ is open, then $C_{\fL}^1 (\Lambda)$ is the collection of functions $h : \Lambda \to \mR$ such that the classical derivative $X h$ exists and is continuous on $\Lambda$ for all $X \in \fl_1$.

The following is Proposition 3.16 of \cite[p.~26]{SerraCassanoNotes}.

\begin{lem} \label{lem:scc1}
Let $\fL$ be a stratified Lie group and let $\Lambda \subset \fL$ be open. A function $h : \Lambda \to \mR$ belongs to $C_{\fL}^1 (\Lambda)$ if and only if the distributional derivative $X h$ is continuous for all $X \in \fl_1$. 
\end{lem}

This lemma plays an important role in the proof of the next result.

\begin{pro} \label{pro:c1vfc}
$(\cP f)_{4,4}, (\cP f)_{5,4}, (\cP f)_{4,5}, (\cP f)_{5,5} \in C_{\fC}^1 (\Upomega)$.
\end{pro}

\begin{proof}
Let $\Upomega_0 \subset\subset \Upomega$ be open. There exists $\epsilon_0 > 0$ such that for each $k=1,\ldots,5$ there is a family $\{f_k^\epsilon \,|\, 0<\epsilon < \epsilon_0 \}$ of $C^\infty$-smooth functions defined on $\Upomega_0$ with (i) $f_k^\epsilon \to f_k|_{\Upomega_0}$ locally uniformly as $\epsilon \to 0$ and (ii) $Y_j f_k^\epsilon \to Y_j f_k |_{\Upomega_0}$ locally uniformly as $\epsilon \to 0$ for all $j = 1,\ldots,5$. This can be achieved using the underlying Euclidean structure, or using a convolution defined in terms of group operations as developed in \cite{HSHG}. In the following we take $0 < \epsilon < \epsilon_0$ always.

Define $\cP_{i,j}^\epsilon$ to be $(\cP f)_{i,j}$ with each instance of the component function $f_k$ replaced with $f_k^\epsilon$. For example,
\[
    \cP_{4,5}^{\epsilon} = Z_2 f_4^\epsilon + \tfrac{1}{2} \left( f_3^\epsilon Z_2 f_1^\epsilon - f_1^\epsilon Z_2 f_3^\epsilon \right) - \tfrac{1}{6}\left(f_1^\epsilon f_2^\epsilon Z_2 f_1^\epsilon - (f_1^\epsilon)^2 Z_2 f_2^\epsilon\right).
\]
Note, $\cP_{i,j}^\epsilon$ converges to $(\cP f)_{i,j}$ locally uniformly so $\cP_{i,j}^\epsilon \to (\cP f)_{i,j}$ in $\cD' (\Upomega_0)$.

We observe that
\begin{align*}
    X_2 \cP_{4,5}^\epsilon = X_2 Z_2 f_4^\epsilon &+ \tfrac{1}{2}\left(X_2 f_3^\epsilon Z_2 f_1^\epsilon + f_3^\epsilon X_2 Z_2 f_1^\epsilon - X_2 f_1^\epsilon Z_2 f_3^\epsilon - f_1^\epsilon X_2 Z_2 f_3^\epsilon \right) \\
    &-\tfrac{1}{6}\left( f_2^\epsilon X_2 f_1^\epsilon Z_2 f_1^\epsilon + f_1^\epsilon X_2 f_2^\epsilon Z_2 f_1^\epsilon + f_1^\epsilon f_2^\epsilon X_2 Z_2 f_1^\epsilon \right)\\
    &+\tfrac{1}{6}\left( 2 f_1^\epsilon X_2 f_1^\epsilon Z_2 f_2^\epsilon + (f_1^\epsilon)^2 X_2 Z_2 f_2^\epsilon \right).
\end{align*}
This is nothing but repeated application of the product rule. Now we compute,
\begin{align*}
    Z_2 \cP_{4,2}^\epsilon = X_2 Z_2 f_4^\epsilon &+ \tfrac{1}{2}\left( X_2 f_1^\epsilon Z_2 f_3^\epsilon + f_3^\epsilon X_2 Z_2 f_1^\epsilon - X_2 f_3^\epsilon Z_2 f_1^\epsilon - f_1^\epsilon X_2 Z_2 f_3^\epsilon \right) \\
    &-\tfrac{1}{6}\left( f_2^\epsilon X_2 f_1^\epsilon Z_2 f_1^\epsilon + f_1^\epsilon X_2 f_1^\epsilon Z_2 f_2^\epsilon + f_1^\epsilon f_2^\epsilon X_2 Z_2 f_1^\epsilon \right) \\
    &+\tfrac{1}{6}\left( 2 f_1^\epsilon X_2 f_2^\epsilon Z_2 f_1^\epsilon + (f_1^\epsilon)^2 X_2 Z_2 f_2^\epsilon \right).
\end{align*}
Here we have used that $Z_2$ is in the center of $\fc$ so that $X_2$ and $Z_2$ commute. It follows,
\begin{align*}
    X_2 \cP_{4,5}^\epsilon - Z_2 \cP_{4,2}^\epsilon &= \tfrac{1}{2}\left(X_2 f_3^\epsilon Z_2 f_1^\epsilon - X_2 f_1^\epsilon Z_2 f_3^\epsilon - X_2 f_1^\epsilon Z_2 f_3^\epsilon + X_2 f_3^\epsilon Z_2 f_1^\epsilon \right) \\
    &\,-\tfrac{1}{6}\left( f_1^\epsilon X_2 f_2^\epsilon Z_2 f_1^\epsilon - f_1^\epsilon X_2 f_1^\epsilon Z_2 f_2^\epsilon \right) \\
    &\,+\tfrac{1}{6}\left( 2 f_1^\epsilon X_2 f_1^\epsilon Z_2 f_2^\epsilon - 2 f_1^\epsilon X_2 f_2^\epsilon Z_2 f_1^\epsilon \right)\\
    &= X_2 f_3^\epsilon Z_2 f_1^\epsilon - X_2 f_1^\epsilon Z_2 f_3^\epsilon +\tfrac{1}{2}\left( f_1^\epsilon X_2 f_1^\epsilon Z_2 f_2^\epsilon - f_1^\epsilon X_2 f_2^\epsilon Z_2 f_1^\epsilon \right).
\end{align*}
This last expression involves only first derivatives of the $f_k^\epsilon$. Hence,
\begin{align*}
    X_2 (\cP f)_{4,5} &= \lim_{\epsilon \to 0} \left(X_2 \cP_{4,5}^\epsilon - Z_2 \cP_{4,2}^\epsilon\right)\\
    &= X_2 f_3 Z_2 f_1 - X_2 f_1 Z_2 f_3 +\tfrac{1}{2}\left( f_1 X_2 f_1 Z_2 f_2 - f_1 X_2 f_2 Z_2 f_1 \right)
\end{align*}
in $\cD'(\Upomega_0)$. Already this shows the distributional derivative $X_2 (\cP f)_{4,5}$ is continuous. Pushing further, we find it admits a cleaner description. Since $(\cP f)_{3,2} = 0$,  
\[
    X_2 f_3 = \tfrac{1}{2}\left( f_1 X_2 f_2 - f_2 X_2 f_1 \right)
\]
and we find
\begin{align*}
    X_2 (\cP f)_{4,5} &= \tfrac{1}{2}\left( f_1 X_2 f_2 - f_2 X_2 f_1 \right) Z_2 f_1 - X_2 f_1 Z_2 f_3 +\tfrac{1}{2}\left( f_1 X_2 f_1 Z_2 f_2 - f_1 X_2 f_2 Z_2 f_1 \right)\\
    &= -X_2 f_1 \left( Z_2 f_3 + \tfrac{1}{2} \left( f_2 Z_2 f_1 - f_1 Z_2 f_2 \right) \right) \\
    &= -X_2 f_1 \left< f^* \theta , Z_2 \right>.
\end{align*}
In a similar way we discover
\begin{equation} \label{eq:nicereps}
\begin{aligned}
    X_1 (\cP f)_{4,4} &= -X_1 f_1 \left< f^* \theta , Z_1 \right>, \\
    X_2 (\cP f)_{4,4} &= -X_2 f_1 \left< f^* \theta , Z_1 \right>, \\
    X_1 (\cP f)_{5,4} &= -X_1 f_2 \left< f^* \theta , Z_1 \right>, \\
    X_2 (\cP f)_{5,4} &= -X_2 f_2 \left< f^* \theta , Z_1 \right>, \\
    X_1 (\cP f)_{4,5} &= -X_1 f_1 \left< f^* \theta , Z_2 \right>, \\
    X_2 (\cP f)_{4,5} &= -X_2 f_1 \left< f^* \theta , Z_2 \right>, \\
    X_1 (\cP f)_{5,5} &= -X_1 f_2 \left< f^* \theta , Z_2 \right>, \text{ and}\\
    X_2 (\cP f)_{5,5} &= -X_2 f_2 \left< f^* \theta , Z_2 \right>.
\end{aligned}
\end{equation}
(We have included the already discussed $X_2 (\cP f)_{4,5}$ for the sake of a complete list.)

That $(\cP f)_{4,4}, (\cP f)_{5,4}, (\cP f)_{4,5}, (\cP f)_{5,5}\in C_{\fC}^1 (\Upomega_0)$ now follows from Lemma \ref{lem:scc1}. The classical first horizontal derivatives are given by list \eqref{eq:nicereps}. As $\Upomega_0$ was an arbitrary open set compactly contained in $\Upomega$, it must be that $(\cP f)_{4,4}, (\cP f)_{5,4}, (\cP f)_{4,5}, (\cP f)_{5,5} \in C_{\fC}^1 (\Upomega)$ as desired.
\end{proof}

Let $\Upsilon = \alpha_{1,1} X_1 + \alpha_{1,2} X_2 + \alpha_2 Y + \alpha_{3,1} Z_1 + \alpha_{3,2}Z_2$ be a generalized vector field on $\Upomega' = f(\Upomega)$. By definition, $\Upsilon$ is contact if 
\begin{gather*}
    X_1 \alpha_2 = -\alpha_{1,2}, \quad X_2 \alpha_2 = \alpha_{1,1}, \\
    X_1 \alpha_{3,1} = \alpha_2,\quad X_1 \alpha_{3,2} = 0,\quad X_2 \alpha_{3,1} = 0,\quad\text{and}\quad X_2 \alpha_{3,2} = \alpha_2.
\end{gather*}

\begin{pro} \label{pro:inducedcontactvf} 
If 
\[
    \alpha = [(J_H f)\circ g] [(X_1 f_1) \circ g] \quad\text{and}\quad \beta = [(J_H f) \circ g ][(X_1 f_2) \circ g]
\]
or
\[
    \alpha = [(J_H f)\circ g] [(X_2 f_1) \circ g] \quad\text{and}\quad \beta = [(J_H f) \circ g ][(X_2 f_2) \circ g]
\]
then
\[
    \Upsilon = (X_2 X_1 \alpha) X_1 - (X_1 X_2 \beta) X_2 + (X_1 \alpha) Y + \alpha Z_1 + \beta Z_2
\]
is a contact generalized vector field on $\Upomega'$.
\end{pro}

\begin{proof}
The form of $\Upsilon$ implies the proof reduces (in either case) to showing both $X_1 \alpha = X_2 \beta$ and $X_1 \beta = 0 = X_2 \alpha$ in $\cD' (\Upomega')$. We work the case of $\alpha = [(J_H f)\circ g] [(X_2 f_1) \circ g]$ and $\beta = [(J_H f) \circ g ][(X_2 f_2) \circ g]$ in detail. 

In the current circumstances, that $(\cP (f \circ g)) = [(\cP f)\circ g](\cP g)$ and that $(\cP g)$ is invertible are basic facts. They allow us to identify
\[
    \alpha = -\frac{X_2 g_1}{J_H^2 g} = -\frac{J_H g (X_2 g_1)}{J_H^3 g} = -\frac{(\cP g)_{4,5}}{(\cP g)_{4,4} (\cP g)_{5,5} - (\cP g)_{4,5}(\cP g)_{5,4}} 
\]  
and
\[
    \beta = \frac{X_1 g_1}{J_H^2 g } = \frac{J_H g (X_1 g_1)}{J_H^3 g} = \frac{(\cP g)_{4,4}}{(\cP g)_{4,4}(\cP g)_{5,5} - (\cP g)_{4,5}(\cP g)_{5,4}}.
\]
We now start writing $\cP_{i,j}$ for $(\cP g)_{i,j}$. By Proposition \ref{pro:c1vfc} and \eqref{eq:nicereps} we have
\begin{align*}
    - X_2 \alpha &= \frac{X_2 \cP_{4,5}(\cP_{4,4}\cP_{5,5} - \cP_{4,5}\cP_{5,4})}{J_H^6 g} \\
    &\quad- \frac{\cP_{4,5}((X_2 \cP_{4,4})\cP_{5,5} + \cP_{4,4}(X_2 \cP_{5,5}) - (X_2\cP_{4,5})\cP_{5,4} - \cP_{4,5}(X_2 \cP_{5,4}))}{J_H^6 g} \\
    &=\frac{(X_2 \cP_{4,5})\cP_{4,4}\cP_{5,5}}{J_H^6 g} - \frac{(X_2 \cP_{4,4})\cP_{4,5}\cP_{5,5}}{J_H^6 g} - \frac{(X_2 \cP_{5,5})\cP_{4,4}\cP_{4,5}}{J_H^6 g} + \frac{(X_2 \cP_{5,4})\cP_{4,5}^2}{J_H^6 g} \\
    &= -\frac{X_2 g_1 \left< g^* \theta , Z_2 \right>\cP_{4,4}\cP_{5,5}}{J_H^6 g} + \frac{X_2 g_1 \left< g^* \theta , Z_1 \right>\cP_{4,5}\cP_{5,5}}{J_H^6 g} \\
    &\quad + \frac{X_2 g_2 \left< g^* \theta , Z_2 \right>\cP_{4,4}\cP_{4,5}}{J_H^6 g} - \frac{X_2 g_2 \left< g^* \theta , Z_1 \right>\cP_{4,5}^2}{J_H^6 g} \\
    &=-\frac{X_1 g_1 X_2 g_1 X_2 g_2 \left< g^* \theta , Z_2 \right>}{J_H^4 g} + \frac{(X_2 g_1)^2 X_2 g_2 \left< g^* \theta , Z_1 \right>}{J_H^4 g} \\
    &\quad + \frac{X_1 g_1 X_2 g_1 X_2 g_2 \left< g^* \theta , Z_2 \right>}{J_H^4 g} - \frac{(X_2 g_1)^2 X_2 g_2 \left< g^* \theta , Z_1 \right>}{J_H^4 g} \\
    &= 0.
\end{align*}
Similar calculations lead to
\begin{align*}
    X_1 \beta &=\frac{X_1 g_1 X_2 g_1 X_1 g_2 \left< g^* \theta , Z_1 \right>}{J_H^4 g} + \frac{(X_1 g_1)^2 X_1 g_2 \left< g^* \theta , Z_2 \right>}{J_H^4 g} \\
    &\quad - \frac{(X_1 g_1)^2 X_1 g_2 \left< g^* \theta , Z_2 \right>}{J_H^4 g} - \frac{X_1 g_1 X_2 g_1 X_1 g_2 \left< g^* \theta , Z_1 \right>}{J_H^4 g} \\
    &= 0.
\end{align*}
Furthermore, it is straightforward to check that
\begin{align*}
    X_1 \alpha &=\frac{(X_1 g_1)^2 X_2 g_2 \left< g^* \theta , Z_2 \right>}{J_H^4 g} - \frac{X_1 g_1 X_2 g_1 X_2 g_2 \left< g^* \theta , Z_1 \right>}{J_H^4 g} \\
    &\quad - \frac{X_1 g_1 X_2 g_1 X_1 g_2 \left< g^* \theta , Z_2 \right>}{J_H^4 g} + \frac{(X_2 g_1)^2 X_1 g_2 \left< g^* \theta , Z_1 \right>}{J_H^4 g} \\
    &= X_2 \beta.
\end{align*}

The case of $\alpha = [(J_H f)\circ g] [(X_1 f_1) \circ g]$ and $\beta = [(J_H f) \circ g ][(X_1 f_2) \circ g]$ is left to the reader.
\end{proof}

We are now ready to prove the main result of this paper. In the proof we write $\Delta_{\fC}$ for the sub-elliptic Laplacian on $\fC$, $\Delta_{\fC} = X_1 X_1 + X_2 X_2$.

\begin{thm}\label{thm:end}
$f \in C^{\infty}(\Upomega)$.
\end{thm}

\begin{proof}
From Proposition \ref{pro:inducedcontactvf} and Proposition \ref{pro:contactvfpoly} we have that $X_j g_i / J_H^2 g$ is a polynomial for each combination of $i,j = 1,2$. Consequently, $J_H^{-3} g$ is a polynomial. Since $0 < J_H g < \infty$ on $\Upomega'$ it follows $J_H g$ is $C^\infty$-smooth on $\Upomega'$. We now see that each $X_j g_i$, $1\leq i,j \leq 2$, is $C^\infty$-smooth. This is enough to conclude that $\Delta_{\fC} g_i$ is $C^\infty$-smooth for each $i=1,2$. Because $\Delta_{\fC}$ is a hypoelliptic operator this gives $g_i$ is $C^\infty$-smooth on $\Upomega'$ for each $i = 1,2$. At this point we invoke that $(\cP g)_{3,1}$, $(\cP g)_{4,1}$, $(\cP g)_{5,1}$, $(\cP g)_{3,2}$, $(\cP g)_{4,2}$, and $(\cP g)_{5,2}$ are all zero to find that $X_j g_i$ is $C^\infty$-smooth for all $j=1,2$ and $i = 3,4,5$. Arguing as before, we conclude that $g_i$ is $C^\infty$-smooth for all $i=3,4,5$, hence for all $i = 1,2,3,4,5$. Clearly, $f$ is $C^\infty$-smooth by the symmetry of the situation.
\end{proof}

\begin{rem}
Use of $\Delta_{\fC}$ in the proof of Theorem \ref{thm:end} could be considered heavy-handed, however, we find it a transparent argument. In any case, it is interesting to note that though $\Delta_{\fC}$ is hypoelliptic by H\"{o}rmander's theorem, it is not analytic-hypoelliptic by \cite{NonAnalyticHypoelliptic}. 
\end{rem}

\bibliographystyle{plain}
\bibliography{aabibdb} 

\def\cprime{$'$} \def\cprime{$'$} \def\cprime{$'$}
\begin{thebibliography}{10}

\bibitem{AustinTyson}
A.~D. Austin and J.~T. Tyson.
\newblock A new proof of the {$C^\infty$} regularity of {$C^2$} conformal
  mappings on the {H}eisenberg group.
\newblock {\em Colloq. Math.}, 150(2):217--228, 2017.

\bibitem{RollingSpinor}
J.~C. Baez and J.~Huerta.
\newblock {$G_2$} and the rolling ball.
\newblock {\em Trans. Amer. Math. Soc.}, 366(10):5257--5293, 2014.

\bibitem{RollingBalls}
G.~Bor and R.~Montgomery.
\newblock {$G_2$} and the rolling distribution.
\newblock {\em Enseign. Math. (2)}, 55(1-2):157--196, 2009.

\bibitem{CapognaRegularity99}
L.~Capogna.
\newblock Regularity for quasilinear equations and {$1$}-quasiconformal maps in
  {C}arnot groups.
\newblock {\em Math. Ann.}, 313(2):263--295, 1999.

\bibitem{Capogna&Cowling}
L.~Capogna and M.~Cowling.
\newblock Conformality and {$Q$}-harmonicity in {C}arnot groups.
\newblock {\em Duke Math. J.}, 135(3):455--479, 2006.

\bibitem{CartanFiveVariables}
E.~Cartan.
\newblock Les syst\`emes de {P}faff, \`a cinq variables et les \'{e}quations
  aux d\'{e}riv\'{e}es partielles du second ordre.
\newblock {\em Ann. Sci. \'{E}cole Norm. Sup. (3)}, 27:109--192, 1910.

\bibitem{NonAnalyticHypoelliptic}
M.~Christ.
\newblock Nonexistence of invariant analytic hypoelliptic differential
  operators on nilpotent groups of step greater than two.
\newblock In {\em Essays on {F}ourier analysis in honor of {E}lias {M}. {S}tein
  ({P}rinceton, {NJ}, 1991)}, volume~42 of {\em Princeton Math. Ser.}, pages
  127--145. Princeton Univ. Press, Princeton, NJ, 1995.

\bibitem{CDMKR}
M.~Cowling, F.~De~Mari, A.~Kor\'{a}nyi, and H.~M. Reimann.
\newblock Contact and conformal maps in parabolic geometry. {I}.
\newblock {\em Geom. Dedicata}, 111:65--86, 2005.

\bibitem{CowOttGlobal}
M.~G. Cowling and A.~Ottazzi.
\newblock Global contact and quasiconformal mappings of {C}arnot groups.
\newblock {\em Conform. Geom. Dyn.}, 19:221--239, 2015.

\bibitem{MissingFromCartan}
B.~Doubrov and A.~Govorov.
\newblock A new example of a generic 2-distribution on a 5-manifold with large
  symmetry algebra.
\newblock {\em arXiv:1305.7297}, 2013.

\bibitem{HSHG}
G.~B. Folland and E.~M. Stein.
\newblock {\em Hardy spaces on homogeneous groups}, volume~28 of {\em
  Mathematical Notes}.
\newblock Princeton University Press, Princeton, N.J.; University of Tokyo
  Press, Tokyo, 1982.

\bibitem{HormanderI}
L.~H\"{o}rmander.
\newblock {\em The analysis of linear partial differential operators. {I}},
  volume 256 of {\em Grundlehren der Mathematischen Wissenschaften [Fundamental
  Principles of Mathematical Sciences]}.
\newblock Springer-Verlag, Berlin, 1983.
\newblock Distribution theory and Fourier analysis.

\bibitem{KMX}
B.~Kleiner, S.~M\"{u}ller, and X.~Xie.
\newblock Pansu pullback and rigidity of mappings between carnot groups.
\newblock {\em arXiv:2004.09271}, 2020.

\bibitem{LiuAnother}
Z.~Liu.
\newblock Another proof of the {L}iouville theorem.
\newblock {\em Ann. Acad. Sci. Fenn. Math.}, 38(1):327--340, 2013.

\bibitem{OttazziWarhurstGeneral}
A.~Ottazzi and B.~Warhurst.
\newblock Contact and 1-quasiconformal maps on {C}arnot groups.
\newblock {\em J. Lie Theory}, 21(4):787--811, 2011.

\bibitem{Sachkov}
Y.~L. Sachkov.
\newblock Symmetries of flat rank two distributions and sub-{R}iemannian
  structures.
\newblock {\em Trans. Amer. Math. Soc.}, 356(2):457--494, 2004.

\bibitem{Sarvas}
J.~Sarvas.
\newblock Ahlfors' trivial deformations and {L}iouville's theorem in {${\bf
  R}^{n}$}.
\newblock In {\em Complex analysis {J}oensuu 1978 ({P}roc. {C}olloq., {U}niv.
  {J}oensuu, {J}oensuu, 1978)}, volume 747 of {\em Lecture Notes in Math.},
  pages 343--348. Springer, Berlin, 1979.

\bibitem{SerraCassanoNotes}
F.~Serra~Cassano.
\newblock Some topics of geometric measure theory in {C}arnot groups.
\newblock In {\em Geometry, analysis and dynamics on sub-{R}iemannian
  manifolds. {V}ol. 1}, EMS Ser. Lect. Math., pages 1--121. Eur. Math. Soc.,
  Z\"{u}rich, 2016.

\bibitem{Tanaka1}
N.~Tanaka.
\newblock On differential systems, graded {L}ie algebras and pseudogroups.
\newblock {\em J. Math. Kyoto Univ.}, 10:1--82, 1970.

\bibitem{WarhurstC1}
B.~Warhurst.
\newblock Contact and {P}ansu differentiable maps on {C}arnot groups.
\newblock {\em Bull. Aust. Math. Soc.}, 77(3):495--507, 2008.

\bibitem{Willse}
T.~Willse.
\newblock Cartan's incomplete classification and an explicit ambient metric of
  holonomy {$\rm G_2^*$}.
\newblock {\em Eur. J. Math.}, 4(2):622--638, 2018.

\bibitem{Yamaguchi}
K.~Yamaguchi.
\newblock Differential systems associated with simple graded {L}ie algebras.
\newblock In {\em Progress in differential geometry}, volume~22 of {\em Adv.
  Stud. Pure Math.}, pages 413--494. Math. Soc. Japan, Tokyo, 1993.

\end{thebibliography}

\end{document}